\documentclass[11pt]{amsart}

\usepackage{amssymb}
\usepackage{enumerate}
\usepackage[dvips]{graphicx}
\DeclareGraphicsExtensions{.eps,.art,.ART,.ps}

\usepackage{amsmath, amsthm, amsfonts, color, dsfont}

\newcommand{\R}{\mathbb{R}}      
\newcommand{\ds}{\displaystyle}

\newcommand{\tr}{\operatorname{tr}}

\newcommand{\dint}{\ds\int}

\newcommand{\dsum}{\ds\sum}
\newcommand{\eps}{\varepsilon}
\newcommand{\eqskip}{ \vspace*{2mm}\\ }

\newtheorem{theorem}{Theorem}
\newtheorem{thm}{Theorem}[section]

\newtheorem{lemma}[thm]{Lemma}

\theoremstyle{definition}

\theoremstyle{remark}

\newtheorem{rmk}[thm]{Remark}

\def\Om{\Omega}
\def\om{\omega}
\def\e{\varepsilon}

\def\p{\partial}
\def\D{\Delta}

\def\a{\alpha}

\def\Si{\Sigma}
\def\d{\delta}

\def\Odr{\mathcal{O}}

\def\di{\,\mathrm{d}}

\def\xm{\overline{x}}
\def\xim{\overline{\xi}}

\def\xmb{\xm_{2}}
\def\da{d_{1}}
\def\pa{p_{1}}

\numberwithin{equation}{section}

\newcommand{\Hm}[1]{\leavevmode{\marginpar{\tiny%
$\hbox to 0mm{\hspace*{-0.5mm}$\leftarrow$\hss}%
\vcenter{\vrule depth 0.1mm height 0.1mm width \the\marginparwidth}%
\hbox to 0mm{\hss$\rightarrow$\hspace*{-0.5mm}}$\\\relax\raggedright
#1}}}

\title[Asymptotics for the expected lifetime of Brownian motion]
{Asymptotics for the expected lifetime of Brownian motion on thin domains in $\mathbb{R}^n$}
\author{Denis Borisov \and Pedro Freitas}
\address{
Department of Physics and Mathematics, Bashkir State Pedagogical
University, October rev. st., 3a, 450000, Ufa, Russia
}\email{borisovdi@yandex.ru}
\address{Department of Mathematics,
Faculdade de Motricidade Humana (TU Lisbon) {\rm and} Group of
Mathematical Physics of the University of Lisbon\\ Complexo
Interdisciplinar, Av.~Prof.~Gama Pinto~2\\ P-1649-003 Lisboa,
Portugal}\email{freitas@cii.fc.ul.pt}
\date{\today}
\subjclass[2000]{Primary 60J65; Secondary 35J05}
\keywords{Brownian motion, exit time, asymptotic expansion, torsion}

\begin{document}

\allowdisplaybreaks \maketitle

\begin{itemize}
 \item[] Denis Borisov (borisovdi@yandex.ru): Department of Physics and Mathematics, Bashkir
State Pedagogical
University, October rev. st., 3a, 450000, Ufa, Russia\vspace*{0.5cm}
\item[]  Pedro Freitas\footnote{Corresponding author} (freitas@cii.fc.ul.pt): Department of Mathematics,
Faculdade de Motricidade Humana (TU Lisbon) {\rm and} Group of
Mathematical Physics of the University of Lisbon, Complexo
Interdisciplinar, Av.~Prof.~Gama Pinto~2, P-1649-003 Lisboa,
Portugal
\end{itemize}
\vspace*{2cm}

\begin{abstract}
We derive a three-term asymptotic expansion for the expected lifetime
of Brownian motion and for the torsional rigidity on thin domains in
$\mathds{R}^{n}$, and a two-term expansion for the maximum (and
corresponding maximizer) of the expected lifetime. The approach is
similar to that which we used previously to study the eigenvalues of
the Dirichlet Laplacian and consists of scaling the domain in one
direction and deriving the corresponding asymptotic expansions as
the scaling parameter goes to zero. Apart from being dominated by
the one-dimensional Brownian motion along the direction of the
scaling, we also see that the symmetry of the perturbation plays
a role in the expansion.

As in the case of eigenvalues, these expansions may also be used to
approximate the exit time for domains where the scaling parameter is
not necessarilly close to zero.
\end{abstract}

\newpage

\section{Introduction}
Let $\Omega$ be a bounded open set in $\R^{n}$ and consider the elliptic
equation
\begin{equation}\label{ellipteq}
\begin{array}{rl}
-\Delta u(x) = 2, & x\in\Omega\eqskip
u(x)=0, & x\in\partial\Omega
\end{array}
\end{equation}
Here and throughout the whole paper we use the notation $\Delta$ for the
second order elliptic operator defined by
\[
\Delta u = \dsum_{i=1}^{n} \frac{\ds \partial^{2}}{\ds \partial x_{i}^2}u.
\]
When the operator is acting only on some of these variables, this will be
indicated by a subscript. Note that $\Delta = 2\Delta^P$, where $\Delta^{P}$ denotes
the probabilistic Laplacian, that is, the generator of the Brownian motion.
Under certain mild conditions on the regularity of the boundary, the above
equation has one and only one non-negative solution in $C(\overline{\Omega})$.
However, except for a few domains such as ellipsoids, the solution is not known in
closed form -- see~\cite{afr} for a recent result in the case of equilateral
triangles.

On the other hand, solutions to equation~(\ref{ellipteq}) have a
probabilistic interpretation in terms of the Brownian motion
associated to the Laplacian in $\R^{n}$. More precisely, the value
of $u$ at each point $x$ yields the expected lifetime of a particle
starting from $x$. This, together with what was mentioned
above regarding the (lack of) existence of explicit solutions, makes
it of interest to be able to determine approximations for the
exit time, where asymptotic approximations then play an important role -- see,
for instance,~\cite{gh,GKP} and, for more recent work along these
lines,~\cite{begk,uchi}.

It is the purpose of the present paper to apply to this problem an
approach that we have used recently in the case of Dirichlet eigenvalue
problems and which provides quite accurate approximations for the first
eigenvalue of the Laplace operator for certain classes of
domains~\cite{bofr1,bofr2}. The idea consists in scaling the domain
$\Omega$ along one direction and determining the asymptotic
expansion of the solution in terms of the scaling parameter as it
goes to zero. In this way we obtain an asymptotic expansion for the
solution of equation~(\ref{ellipteq}) from which it is possible to
derive expansions for other quantities such as the maximum of $u$,
the corresponding maximizer and also the integral of $u$. In line
with the above interpretation, the second of these quantities
corresponds to the point in the domain with the largest expected
lifetime, while the last one is known in the elasticity literature
as the torsional rigidity -- see~\cite{bbc} and the references
therein.

Due to the way in which the perturbation is set up it is possible
to reduce the problem of finding approximations of solutions of
equation~(\ref{ellipteq}) to a sequence of one dimensional problems
which may be solved explicitly. Because of this, we should then expect
the expansions in question to be dominated by a term corresponding
to a one-dimensional Brownian motion in the direction along which
the scaling is being performed. This is indeed the case and we see
that the first term in all these expansions does correspond to the
solution of the respective one-dimensional problem -- expected
lifetime, maximum expected lifetime or torsion -- on an interval
whose length is the maximum of the height function along this
direction.

The second term in the asymptotic expansions, on the other hand,
has a more geometric interpretation as it also
depends on a function that measures how asymmetric the domain is
with respect to the hyperplane orthogonal to the scaling direction.
In particular, we see that for a given height function,
quantities such as the maximum expected lifetime or the torsion of thin
domains are maximal if the domain is as symmetric as possible in this
direction -- see Theorems~\ref{th1.1},~\ref{maxthm}
and~\ref{torsrig} and the remarks following them.

As may be seen from the examples in Section~\ref{sexamp}, although our
approximations are quite accurate for a fairly large range of values of
the scaling parameter, the error can vary a lot as this parameter approaches
one, depending on the domain under consideration -- see the discussion in
that section in term of the radius of convergence of the corresponding
series.

To the best of our knowledge, the only similar situation that has been
addressed previously in the literature was the case of thin tubular
neighbourhoods (of constant section) of a compact submanifold of a
Riemannian manifold. In this case, a two-term expansion was given
in~\cite{GKP} for the asymptotic behaviour of the mean exit time as
the width of the tube approached zero.

It should also be mentioned that there is a large number of
papers devoted to the asymptotic expansions of the solutions to
elliptic boundary value problems in thin domains -- see \cite{Na},
\cite{NT} and the references therein. However, and to the best
of our knowledge, the problem considered here has not been studied
from this point of view.

The paper is organized as follows. In the next section we lay down
the notation and state the main results of the paper. In
Section~\ref{spthm1} we prove the asymptotic expansion for solutions
of the elliptic problem depending on a scaling parameter.
Sections~\ref{maxproof} and~\ref{storsrig} then present the
asymptotics for the maximum of this solution and for the torsional
rigidity, respectively. We finish with an analysis of the error of
these approximations for some specific domains.

\section{Formulation of the problem and main results}

Let $x=(x',x_n)$, $x'=(x_1,\ldots,x_{n-1})$ be Cartesian coordinates
in $\mathds{R}^n$ and $\mathds{R}^{n-1}$, respectively, $\om$ be a
bounded domain in $\mathds{R}^{n-1}$ with $C^1$-boundary. By
$h_\pm=h_\pm(x')\in C(\overline{\om})\cap C^1(\om)$ we denote two
arbitrary functions such that
$H(x'):=h_-(x')+h_+(x')>0$ for $x'\in\om$ and $H(x')=0$ on $\p\om$.
We also define the two functions
\[
d(x') = h_+(x')- h_-(x') \mbox{ and } p(x') = h_+(x')h_-(x').
\]
Note that due to the relation $4p(x')=H^{2}(x')-d^{2}(x')$ it is possible
to interchange the functions in terms of which our results are presented.
In general we chose those combinations which allowed to write the results
in the most compact way. However, while the geometric interpretation of
the $H$ and $d$ terms is quite clear, and in which in particular the
function $d$ is a measure of the symmetry of the domain in the scaling
direction, the meaning of the function $p$ is not as straightforward.

We now introduce a thin domain by
\begin{equation*}
\Om_\e:=\{x: x'\in\om, -\e h_-(x')<x_n<\e h_+(x')\},
\end{equation*}
where $\e$ is a small parameter, and consider the problem
\begin{equation}\label{1.1}
-\D u^\e=2 \quad \text{in}\quad \Om_\e,\qquad u^\e=0\quad
\text{on}\quad \p\Om_\e.
\end{equation}
In view of the smoothness of the functions
$h_\pm$ and the boundary of $\om$ the domain $\Om_\e$ satisfies the
exterior sphere condition at every boundary point. Theorem 6.13
in~\cite[Ch. 6, Sec. 6.3]{GT} implies that the solution to the
problem (\ref{1.1}) belongs to $C(\overline{\Om_\e})$.

Assuming that the functions $h_\pm$ are smooth enough, we introduce
a sequence of functions
\begin{align}
&\a_0^{(2)}(x'):=p(x'),\quad \a_1^{(2)}(x'):=d(x'),\quad
\a_2^{(2)}(x'):=-1,
\label{1.5}
\\
&\a_i^{(2j)}(x'):=-\frac{\ds 1}{\ds i(i-1)}\D_{x'}
\a_{i-2}^{(2j-2)}(x'),\quad i\geqslant 2,\label{1.6}
\\
&\a_1^{(2j)}(x'):=-\sum\limits_{i=2}^{2j-1}
\a_i^{(2j)}(x') \sum\limits_{m=0}^{i-1} \big(h_+(x')\big)^m
\big(-h_-(x')\big)^{i-m-1},\label{1.7}
\\
&\a_0^{(2j)}(x'):=\sum\limits_{i=2}^{2j-1} \a_i^{(2j)}(x')
\sum\limits_{m=1}^{i-1} \big(h_+(x')\big)^m
\big(-h_-(x')\big)^{i-m}.\label{1.8}
\end{align}

The first of our results describes the uniform asymptotic expansion for
$u^\e$ and forms the basis for the remaining expansions in the paper.

\begin{theorem}\label{th1.1}
Given any $N\geqslant 1$, let the functions $h_\pm$ be such that\linebreak
$\a_i^{(2j)}\in C^2(\overline{\om})$, $j\leqslant N$. Then the
function $u^\e$ satisfies the asymptotic formula
\begin{align}\label{1.9}
&u^\e(x)= \sum\limits_{j=1}^{N} \e^{2j}
u_{2j}\left(x',\frac{x_n}{\e}\right) + \Odr(\e^{2N+2}),
\\
&u_{2j}(x',\xi_n):=\sum\limits_{i=0}^{2j-1}\a_i^{(2j)}(x')\xi_n^i,\label{1.3}
\end{align}
in the $C(\overline{\Om}_\e)$-norm. In particular,
\begin{equation}
\begin{array}{lll}
u_2(x',\xi_n)&=&
-\xi_n^2+\xi_n d(x')+p(x'),
\eqskip
u_4(x',\xi_n)&= &
-\frac{\ds\xi_n^3}{\ds 6}\D_{x'}d(x')-\frac{\ds \xi_n^2}{\ds 2}\D_{x'}
p(x') +\a_1^{(4)}(x')\xi_n+\a_0^{(4)}(x'),\eqskip
u_{6}(x',\xi_{n}) & = &
\frac{\ds \xi_{n}^{5}}{\ds 120}\D_{x'}^{2} d(x')+ \frac{\ds \xi_{n}^{4}}{\ds 24}\D_{x'}^{2} p(x')\eqskip
& & \hspace*{0.2cm}-\frac{\ds \xi_{n}^{3}}{\ds 36}\D_{x'}\left\{\left[d^{2}(x')+p(x')\right]\D_{x'}d(x')+
3d(x')\D_{x'}p(x')\right\}\eqskip
& & \hspace*{0.4cm} -\frac{\ds \xi_{n}^{2}}{\ds 12}\D_{x'}\left[ d(x')p(x')\D_{x'}d(x')+3p(x')\D_{x'}p(x')\right]\eqskip
& &  \hspace*{0.6cm} +\alpha^{(6)}_{1}(x')\xi_{n}+\alpha^{(6)}_{0}(x'),
\end{array}
\label{1.10}
\end{equation}
where
\[
\begin{array}{lll}
\a_1^{(4)}(x')&=&\frac{\ds 1}{\ds 6}
\left[d^2(x')+p(x')\right]\D_{x'} d(x') +\frac{\ds 1}{\ds 2}
d(x')\D_{x'} p(x'), \eqskip \a_0^{(4)}(x')&=&\frac{\ds 1}{\ds
6}d(x')p(x')\D_{x'} d(x') +\frac{\ds 1}{\ds 2} p(x')\D_{x'} p(x'),
\eqskip \a_1^{(6)}(x')&=&\frac{\ds 1}{\ds
12}d(x')\D_{x'}\left[d(x')p(x')\D_{x'}d(x')+3p(x')\D_{x'}p(x')\right]
\eqskip & & \hspace*{0.2cm}+\frac{\ds 1}{\ds
36}\left[d^{2}(x')+p(x')\right]\D_{x'}\left\{
3d(x')\D_{x'}p(x')+\left[d^{2}(x')+p(x')\right]\D_{x'}d(x')\right\}
\eqskip & & \hspace*{0.4cm}-\frac{\ds 1}{\ds
24}d(x')\left[d^{2}(x')+2p(x')\right]\D^{2}_{x'}p(x') \eqskip & &
\hspace*{0.6cm}-\frac{\ds 1}{\ds
120}\left[d^{4}(x')+3d^{2}(x')p(x')+p^{2}(x')\right]\D_{x'}^{2}d(x'),
\eqskip \a_0^{(6)}(x')&=&\frac{\ds p(x')}{\ds 12}\D_{x'}\left[
d(x')p(x')\D_{x'}d(x')+3p(x')\D_{x'}p(x')\right] \eqskip & &
\hspace*{0.2cm}+\frac{\ds d(x')p(x')}{\ds
36}\D_{x'}\left\{\left[d^{2}(x')+p(x')\right]\D_{x'}d(x')+
3d(x')\D_{x'}p(x')\right\} \eqskip & & \hspace*{0.4cm}-\frac{\ds 1
}{\ds 24}p(x')\left[d^{2}(x')+p(x')\right]\D_{x'}^{2} p(x') \eqskip
& & \hspace*{0.6cm}-\frac{\ds 1}{\ds
120}d(x')p(x')\left[d^{2}(x')+2p(x')\right]\D_{x'}^{2} d(x').
\end{array}
\]
\end{theorem}

The coefficients of the asymptotic expansion (\ref{1.9}) involve two
scales, namely, the variable $x'$ and the rescaled variable
$\xi_n:=x_n/\e$, so, this is a two-scale asymptotics. This is a very
natural situation for the problem (\ref{1.1}) since by passing to
the variable $\xi_n$ we get a bounded domain
\begin{equation*}
\Om:=\{(x',\xi_n): x'\in\om, -h_-(x')<\xi_n<h_+(x')\}.
\end{equation*}
In this domain there is no distinguished variable as was the case of $x_n$ in the
domain $\Om_\e$. This is one reason why the asymptotics for $u^\e$
involve two scales. Another way of understanding this fact is that
while $x_n$ ranges in a small interval the remaining variable $x'$
ranges in a bounded set.
 From this point of view, it is natural to rescale the variable $x_n$
and pass to $\xi_n$.

Let us discuss the probabilistic meaning of the first terms in the asymptotic expansion~(\ref{1.9}).
As we will see in the proof of Theorem~\ref{th1.1}, the first term $u_2$ solves the boundary value
problem (\ref{3.3}) below. In view of Remark~8.7b) in \cite[Ch. 8, Sec. 8.1]{MP} the function $u_2(x',\xi_n)$
describes the Brownian motion on the interval $(-h_-(x'), h_+(x'))$ and it is the expected lifetime
for the mentioned segment for a Brownian motion which started at the point $\xi_{n}$.

It is also possible to give a probabilistic-geometric interpretation of the next term $u_4$. This
will be the solution to the boundary value problem (\ref{3.4}) below, when $j=1$, namely,
\begin{align*}
-&\frac{\p^2 u_4}{\p\xi_n^2} =\xi_n\D_{x'}d(x')+\D_{x'}p(x'),\quad
\xi\in\big(-h_-(x'),h_+(x')\big),
\\
& u_4=0,\quad \xi_n=\pm h_\pm(x'),\quad x'\in\om.
\end{align*}
Again by Remark~8.7a) in \cite[Ch. 8, Sec. 8.1]{MP} the function $u_4$ can be represented as 
\begin{equation*}
u_4(x',\xi_n)=\frac{1}{2}\D_{x'}d(x')
\mathbb{E}_{\xi_n} \left[\int\limits_0^T B(t)\di t\right] +\frac{1}{2}\D_{x'}p(x')  \mathbb{E}_{\xi_n}[T],
\end{equation*}
where $\{B(t): t\geqslant 0\}$ is the one-dimensional Brownian motion on the segment $(-h_-(x'),h_+(x'))$, $T:=\inf\{t: t>0, B(t)\not\in(-h_-(x'),h_+(x')\}$, $\mathbb{E}_{\xi_n}$ is the expectation associated with the probability measure $\mathbb{P}_{\xi_n}$ such that
the process $\{B(t): t\geqslant 0\}$ is a Brownian motion started in $\xi_n$. The term 
\begin{equation*}
\frac{1}{2} \mathbb{E}_{\xi_n}[T]=u_2(x')
\end{equation*}
is exactly the lifetime for this one-dimensional process, while
\begin{equation*}
\mathbb{E}_{\xi_n} \left[\int\limits_0^T B(t)\di t\right] 
\end{equation*}
is known in the literature as the occupation time for the Brownian motion which started at the point
$\xi_{n}$ and left the interval at time $T$. The factors $\D_{x'}p(x')$ and $\D_{x'}d(x')$ then
represent a geometric measure of the deviation from the symmetric domain, as mentioned in the Introduction.
More precisely, if the domain is symmetric with respect to the hyperplane where $\om$ lies, then the
difference function $d$ vanishes, $p=H^2/4$ and $u_{4}$
reduces to the expected (one-dimensional) lifetime, affected by a factor which is proportional to the
Laplacian of the square of the height function.

Continuing in the same way, that is, employing equations (\ref{3.4}) for $u_{2j}$ and
\cite[Ch. 8, Sec. 8.1, Rem. 8.7a)]{MP}, it is possible to give similar interpretations for all
other terms in the asymptotics (\ref{1.9}).

From Theorem~\ref{th1.1}, we are then able to derive explicit asymptotic formulas for both the maximum
of $u^{\eps}$ and the torsional rigidity for the family of domains~$\Omega_{\eps}$ as $\eps$ goes
to zero.

\begin{theorem}\label{maxthm}
Let the family of domains $\Omega_{\eps}$ and the functions $h_{-},h_{+}$
and $H$ be as above. We assume further that $H$ satisfies the following hypotheses:
\begin{enumerate}
\item[{\rm H1}] There exists a unique point $\xm\in\omega$ at which $H$
achieves its global maximum which will be denoted by $H_{0}$;
\item[{\rm H2}] The Hessian matrix of $H$ at $\xm$, denoted by $2H_{2}$,
is negative definite;
\item[{\rm H3}] The functions $h_\pm$ are 5 times continuously
differentiable in a vicinity of $\xm$.
\end{enumerate}
Then the maximum value of $u^{\eps}$ in $\Omega_{\eps}$
satisfies
\[
\max_{x\in\Omega_{\eps}} u^{\eps}(x) = \frac{\ds 1}{\ds
4}H_{0}^{2}\eps^{2} + \frac{\ds 1}{\ds 8}H_{0}^2\left[ \frac{\ds 1}{\ds 2}\Delta(H_{0}^2)-|\nabla d_{0}|^{2}
\right]\eps^{4} +\Odr(\e^5),
\]
as $\eps\to0$. Here $\Delta(H_{0}^2)$ denotes Laplacian of the height function squared at the
point of maximum $\xm$, while $\nabla d_{0}$ is the value of the gradient of $d$ at the same point.
\end{theorem}

\begin{rmk}
As mentioned in the Introduction, the first term in the expansion of the maximum corresponds
to that of a one--dimensional Brownian motion on an interval of length $\e H_{0}$. The second term,
on the other hand, has a geometrical interpretation and measures the (local) asymmetry of the domain
in the direction in which the scaling is being carried out, in a neighbourhood of the point of maximum
height. We note that this coefficient will be maximal when $\xm$ is a critical point of the difference
function $d$.
\end{rmk}
\begin{rmk}
If one drops the hypothesis that $H_{2}$ is nonsingular it will still be
possible to obtain an expansion, but this will be much more involved
and will depend on higher order terms in the expansion of $H$ around
$\xm$.
\end{rmk}
\begin{rmk}
The hypothesis on a unique maximum of $H$ may also be dropped and provided there
is only a finite number of such maxima the results still hold except that one has
to construct different expansions for each maximum. The case of a continuum of
maxima is also possible to handle with the techniques employed here but requires
some further changes to the approach.
\end{rmk}
\begin{rmk}
In the process of proving the above theorem we also obtain a two-term asymptotic
expansion for the maximizer -- see Theorem~\ref{xm2xin2}. However, this depends on higher
order terms in the expansions of both $d$ and $p$ and requires the introduction of more
detailed notation which we postpone till Section~\ref{maxproof} below.
\end{rmk}
\begin{rmk}
Under the hypotheses H1 and H2 of Theorem~\ref{maxthm} and assuming
that the functions $h_\pm$ are smooth enough in a vicinity of $\xm$,
it is possible to construct more terms in the asymptotic expansions
for the maximum of $u^\e$ in $\Om_\e$ and for the corresponding
maximizer. In order to do this, one should follow the main lines of
the proof of Theorem~\ref{maxthm}, employing Lemma~\ref{lemmaeps2}
as a starting point. At the same time, it requires bulky and
technical calculations which we would like to avoid. This is the
reason why we provide only two-term asymptotics in
Theorem~\ref{maxthm}.
\end{rmk}

Finally, the integration of the asymptotic expansion for $u^{\eps}$ given by
Theorem~\ref{th1.1} yields
the corresponding asymptotic expansion for the torsional rigidity.
\begin{theorem}[Torsional rigidity]\label{torsrig}
Under the conditions of Theorem~\ref{th1.1} we have
\begin{align*}
\dint_{\Omega_{\eps}} u^{\eps}(x) \;{\rm d}x = & \frac{\ds
\eps^{3}}{\ds 6}\dint_{\omega} H^{3}(y)\;{\rm d}y
\\
 & \hspace*{0.25cm}+ \frac{\ds \eps^{5}}{\ds 24}\dint_{\omega}
H^{3}(y)\left[ \frac{\ds 1}{\ds 2}\Delta_{y}[H^{2}(y)]-|\nabla_{y} d(y)|^2  \right] \;{\rm d}y
\\
 & \hspace*{0.5cm}+\eps^{7}\dint_{\omega} H(y)\left\{\frac{\ds 1}{\ds
720}\left[H^2(y)d^3(y)+3d(y)p^2(y)\right]\D_{y}^{2} d(y)\right.
\\
 & \hspace*{0.75cm}+ \frac{\ds 1}{\ds
120}\left\{H^2(y)d^2(y)+p(y)\left[p(y)-d^2(y)\right]\right\}
\D_{y}^{2} p(y)
\\
&  \hspace*{0.75cm}-\frac{\ds 1}{\ds
144}\left[H^2(y)d(y)-2d(y)p(y)\right]
\\
 &
\hspace*{1cm}\times\D_{y}\left\{\left[d^{2}(y)+p(y)\right]\D_{y}d(y)+
3d(y)\D_{y}p(y)\right\}
\\
&  \hspace*{0.75cm}-\frac{\ds 1}{\ds
36}\left[H^2(y)-3p(y)\right]
\\
& \hspace*{1cm}\times\D_{y}\left[ d(y)p(y)\D_{y}d(y)+3p(y)
\D_{y}p(y)\right]
\\
&  \left.\hspace*{0.75cm}+ \frac{\ds 1}{\ds 2}d(y) \alpha_{1}^{(6)}(y)
+ \alpha^{(6)}_{0}(y)\right\} \;{\rm d}y
\\
 & \hspace*{1cm}+\Odr(\eps^{8}),
\end{align*}
where $\alpha^{(6)}_{0}$ and $\alpha_{1}^{(6)}$ are as in
Theorem~\ref{th1.1}.
\end{theorem}
\begin{rmk}
Again we see that the first term in the expansion corresponds to the one--dimensional Brownian
motion on the line segment with maximal height in the direction of scaling. Also as before, the second
term measures the degree of symmetry of the domain with respect to the hyperplane orthogonal to this direction,
and we see that, for a given height function, this term is maximal when the difference function
vanishes that is, when the domain is symmetric with respect to this hyperplane.
\end{rmk}
\begin{rmk}
In order to obtain the average expected lifetime it remains to divide by the volume
of $\Omega_{\eps}$ which is given by
\[
|\Omega_{\eps}| = \dint_{\Omega_{\eps}} \;{\rm d}x = \dint_{\omega} \dint_{-\eps h_{-}(x')}^{\eps h_{+}(x')}
\;{\rm d}\xi_{n}\;{\rm d}x' = \eps\dint_{\omega} H(y) \;{\rm d}y.
\]

\end{rmk}

\section{The asymptotic expansion for $u^\e$\label{spthm1}}

In this section we prove Theorem~\ref{th1.1}. We begin by passing to the
variables $(x',\xi_n)$ in (\ref{1.1}) leading us to
\begin{equation}\label{3.1}
\left(-\e^2\D_{x'}-\frac{\p^2}{\p \xi_n^2}\right)u^\e=
2\e^2\quad \text{in}\quad\Om,\qquad u^\e=0\quad\text{on}\quad\p\Om.
\end{equation}
We construct the asymptotic expansion to the problem (\ref{3.1}) as
follows
\begin{equation}\label{3.2}
u^\e(x)=\sum\limits_{j=0}^{\infty} \e^{2j} u_{2j}(x',\xi_n),
\end{equation}
where $u^j(x',\xi_n)$ are functions to be determined.

We substitute the expansion (\ref{3.2}) into (\ref{3.1}) and equate
the coefficients of like powers in $\e$. This yields the following boundary
value problems for $u_{2j}$:
\begin{align}
&
\begin{aligned}
-&\frac{\p^2 u_2}{\p\xi_n^2}=2,\quad
\xi\in\big(-h_-(x'),h_+(x')\big),
\\
&u_2=0,\quad \xi_n=\pm h_\pm(x'),\quad x'\in\om,
\end{aligned} \label{3.3}
\\
&
\begin{aligned}
-&\frac{\p^2 u_{2j}}{\p\xi_n^2}=\D_{x'} u_{2j-2},\quad
\xi\in\big(-h_-(x'),h_+(x')\big),
\\
& u_{2j}=0,\quad \xi_n=\pm h_\pm(x'),\quad x'\in\om,\quad j\geqslant
2.
\end{aligned} \label{3.4}
\end{align}

It is easy to check that the solution to (\ref{3.3}) is
\begin{equation*}
u_2(x',\xi_n')=-\xi_n^2+d(x')\xi_n+p(x')
\end{equation*}
that proves (\ref{1.3}) for $j=1$. Substituting this formula into
the (\ref{3.4}) for $j=2$, we get
\begin{equation}
\begin{aligned}
-&\frac{\p^2 u_{4}}{\p\xi_n^2}=\xi_n\D_{x'} d(x')+\D_{x'}p(x'),\quad
\xi\in\big(-h_-(x'),h_+(x')\big),
\\
& u_{4}=0,\quad \xi_n=\pm h_\pm(x'),\quad x'\in\om.
\end{aligned}\label{3.9}
\end{equation}
The solution to the obtained equation is
\begin{equation*}
u_4(x',\xi_n)=\frac{\xi_n^3}{6}\D_{x'}d(x')+\frac{\xi_n^2}{2}\D_{x'}
p(x') +C_1^{(4)}(x')\xi_n+C_0^{(4)}(x'),
\end{equation*}
where $C_i^{(4)}$ are arbitrary functions. We determine them by the
boundary conditions in (\ref{3.9}) and it implies (\ref{1.3}) for
$j=2$.

We prove the remaining formulas (\ref{1.3}) by induction. Assuming
that they are valid for $j\leqslant k$, we consider the equation in
(\ref{3.4}) for $j=k+1$ and see that its general solution reads as
follows,
\begin{equation}\label{3.8}
\begin{aligned}
u_{2k+2}(x',\xi_n)=&-\sum\limits_{i=0}^{2k-1}
\frac{\xi_n^{i+2}}{i(i+2)}\D_{x'} \a_i^{(2k-2)} +C_1^{(2k+2)} \xi_n+
C_0^{(2k+2)}(x')
\\
=&\sum\limits_{i=2}^{2k+1} \a_i^{(2k+1)}\xi_n^i +C_1^{(2k+2)} \xi_n+
C_0^{(2k+2)}(x'),
\end{aligned}
\end{equation}
where $C_i^{(2k+2)}$ are arbitrary functions. The boundary
conditions in (\ref{3.4}) imply the equations for $C_i^{(2k+2)}$,
\begin{align*}
&C_1^{(2k+2)}h_+ + C_0^{(2k+2)}=-\sum\limits_{i=2}^{2k+2}
\a_i^{(2k+2)}h_+^i,
\\
-&C_1^{(2k+2)}h_- + C_0^{(2k+2)}=-\sum\limits_{i=2}^{2k+2}
\a_i^{(2k+2)}(-h_-)^i.
\end{align*}
We solve it and get,
\begin{align*}
C_1^{(2k+2)}=&-\sum\limits_{i=2}^{2k+2}\a_i^{(2k+2)}\frac{h_+^i-(-h_-)^i}{h_+
+ h_-}
\\
=&-\sum\limits_{i=2}^{2k+1} \a_i^{(2j)}(x') \sum\limits_{m=0}^{i-1}
\big(h_+(x')\big)^m \big(-h_-(x')\big)^{i-m-1}=\a_1^{(2k+2)},
\\
C_0^{(2k+2)}=&-\sum\limits_{i=2}^{2k+2} \a_i^{(2k+2)} \left(h_+^i-
\sum\limits_{m=0}^{i-1} h_+^{m+1} (-h_-)^{i-m-1}\right)
\\
=&\sum\limits_{i=2}^{2k+1} \a_i^{(2j)}(x') \sum\limits_{m=1}^{i-1}
\big(h_+(x')\big)^m \big(-h_-(x')\big)^{i-m}=\a_0^{(2k+2)}.
\end{align*}
We substitute the obtained identities into (\ref{3.8}) and arrive at
(\ref{1.3}) for $j=k+1$.


Given any $N\geqslant 1$, assume that $\a_i^{(2j)}\in
C^2(\overline{\om})$, $j\geqslant N$. Let
\begin{equation}\label{3.7}
u^{\e,N}(x',\xi_n):=\sum\limits_{j=0}^{N} \e^{2j} u_{2j}(x',\xi_n).
\end{equation}
It follows from the problems (\ref{3.3}), (\ref{3.4}) that the
function $u^{\e,N}$ solves the boundary value problem
\begin{equation*}
\left(-\e^2\D_{x'}-\frac{\p^2}{\p \xi_n^2}\right)u^{\e,N}=
2\e^2+\e^{2N+2}\D_{x'}u_{2N}\quad \text{in}\quad\Om,\qquad
u^{\e,N}=0\quad\text{on}\quad\p\Om.
\end{equation*}
Hence, the function $\widetilde{u}^{\e,N}:=u^\e-u^{\e,N}$ is the
solution to
\begin{equation}\label{3.5}
\left(\e^2\D_{x'}+\frac{\p^2}{\p\xi_n^2}\right)\widetilde{u}^{\e,N}=
\e^{2N+2}\D_{x'}u_{2N}\quad \text{in}\quad \Om, \qquad
\widetilde{u}^{\e,N}=0\quad\text{on}\quad\p\Om.
\end{equation}
The coefficient affecting the derivative $\frac{\p^2}{\p\xi_n^2}$ in the
last equation is one. Employing this fact and applying the maximum
principle in the form of inequality~(1.9) in \cite[Ch. 3, Sec.
1]{LU}, we obtain the estimate
\begin{equation}\label{3.6}
\|\widetilde{u}^{\e,N}\|_{C(\overline{\Om})} \leqslant C\e^{2N+2}
\|\D_{x'}u_{2N}\|_{C(\overline{\Om})}\leqslant C\e^{2N+2},
\end{equation}
where the constant $C$ is independent of $\e$. It proves the formula
(\ref{1.9}). The formulas (\ref{1.10}) follow directly from
 (\ref{1.5}), (\ref{1.6}), (\ref{1.7}), (\ref{1.8}),  (\ref{1.3}).
The proof is complete.

\section{Proof of Theorem~\ref{maxthm}\label{maxproof}}

In the whole of this section we shall consider the function
$u^{\eps}(x',\xi_{n})$. This is then defined on $\Omega$
and it is clear that it is sufficient to find the maximum of
$u^{\eps}(x',\xi_{n})$ since after rescalling $x_{n}\to x_{n}\eps$
the maximum of the function remains unaltered.

\subsection{Existence of an expansion and terms of order $\eps^{2}$}
We begin by showing that, under the hypothesis of Theorem~\ref{maxthm} and
up to order $\eps^{2}$, the maximum of $u$ has an asymptotic expansion
that may be obtained directly from the expression of $u_{2}$.
\begin{lemma}\label{lm4.1}
Under the hypothesis of Theorem~\ref{maxthm} we have
\[
\max_{x\in\Omega_{\eps}}u^{\eps}(x) =
\max_{x\in\Omega}u_{2}(x)\eps^{2} + \Odr(\eps^{4}) \mbox{ as }
\eps\to 0.
\]
Furthermore
\[
\max_{x\in\Omega}u_{2}(x) = u_{2}\left(\xm, \frac{d(\xm)}{2}\right)
= \frac{\ds 1}{\ds 4}H_{0}^{2}
\]
and is unique.
\end{lemma}
\begin{proof}
The first part follows directly from the asymptotics of $u^{\eps}$ given by Theorem~\ref{th1.1},
since we now have
\[
u^{\eps}(x',\xi_{n}) = \eps^{2} u_{2}(x',\xi_{n}) + \Odr(\eps^{4}).
\]
For the second part, note that we may write
\[
\begin{array}{lll}
u_{2}(x',\xi_{n}) & = & -\left[
\xi_{n}^2+\xi_{n}\left(h_{-}(x')-h_{+}(x')\right)+h_{-}(x')h_{+}(x')\right]\eqskip
& = & -\left[\xi_{n}-\frac{\ds 1}{\ds 2}d(x)\right]^2+\frac{\ds 1}{\ds 4}
H^{2}(x).
\end{array}
\]
This last expression is clearly maximized when $\xi_{n} = d(x)/2$ and $H$ is also maximized,
yielding $x'=\xm$ and $\xi=d(\xm)/2$. The uniqueness follows directly from hypothesis H1.
\end{proof}

In order to go on to obtain the next terms in the expansion for the maximum (and the corresponding
maximizer), we need to show the existence of such an expansion which we do in the next
lemma. This also proves that the coefficients of the terms of order one in both the expansions
for $x'$ and $\xi_{n}$ vanish.

\begin{lemma}\label{lemmaeps2}
Given any $N\geqslant 1$, assume that the functions $h_\pm$ are
$[N/2]+2N$ times continuously differentiable in a vicinity of the
point $\xm$, and the hypotheses {\rm H1} and {\rm H2} of
Theorem~\ref{maxthm} hold true.  Then the function $u^{\e,N}$ has
only one stationary point which is a maximum. The corresponding
maximizer has the following asymptotic expansion
\begin{equation}\label{4.11}
\begin{aligned}
&\xm^{\e,N}=\xm+\sum\limits_{i=1}^{[N/2]} \e^{2i}
\xm_{2i}+\Odr(\e^{2[N/2]+2}),
\\
&\xim^{\e,N}=\frac{d(\xm)}{2}+\sum\limits_{i=1}^{[N/2]} \e^{2i}
\xi_{ni}+\Odr(\e^{2[N/2]+2}).
\end{aligned}
\end{equation}
The maximum of the function $u^\e$ satisfies the identity
\begin{equation}\label{4.12}
\max\limits_{\Om_\e} u^\e(x)=\max\limits_{\Om} u^\e(x',\e\xi_n)=
\max\limits_{\Om} u^{\e,N}(x',\xi_n)+\Odr(\e^{2N+2}),
\end{equation}
and any maximizer $(\xm^\e,\xim^\e)$ of this function has the
asymptotic expansion
\begin{equation}\label{4.11a}
\begin{aligned}
&\xm^\e=\xm+\sum\limits_{i=1}^{[N/2]} \e^{2i}
\xm_{2i}+\Odr(\e^{N+1}),
\\
&\xim^\e=\frac{d(\xm)}{2}+\sum\limits_{i=1}^{[N/2]} \e^{2i}
\xi_{ni}+\Odr(\e^{N+1}).
\end{aligned}
\end{equation}

\end{lemma}
\begin{proof}
The identity (\ref{4.12}) follows directly from the asymptotics
(\ref{1.9}) for $u^\e$.

Let us find the maximum of $u^{\e,N}$. In order to do this, we should
first find the stationary points of this functions by solving the
equation
\begin{equation*}
\nabla_{(x',\xi_n)} u^{\e,N}(x',\xi_n)=0,
\end{equation*}
which is equivalent to
\begin{equation}\label{4.10}
\sum\limits_{i=1}^{N} \e^{2i-2} \nabla_{(x',\xi_n)}
u_{2i}(x',\xi_n)=0.
\end{equation}
It follows from Lemma~\ref{lm4.1} that for $\e=0$ this equation has
a unique solution $(\xm,d(\xm)/2)$. In order to solve (\ref{4.10})
for $\e>0$ we apply the implicit function theorem considering
$(x',\xi_n)$ as functions of $\e$. We first need to check that the
corresponding Jacobian is non-zero. It is easy to see that this
Jacobian at $\e$ equal to zero coincides with the determinant of
$2H_2$ which is non-zero by hypothesis~H2.

The assumption for the smoothness of $h_\pm$ and the formulas
(\ref{1.3}),  (\ref{1.5}), 
(\ref{1.6}), (\ref{1.7}), (\ref{1.8}) for $u_{2i}$ yield that
$u_{2i}(x',\xi_n)$, $i=1,\ldots,N$, are $[N/2]+2$ times continuously
differentiable in a small vicinity of $(\xm,d(\xm)/2)$. The
dependence of the left hand side of (\ref{4.10}) on $\e^2$ is
holomorphic and by the implicit function theorem we conclude that
for $\e$ small enough there exists a unique solution
$(\xm^{\e,N},\xim^{\e,N})$ to (\ref{4.10}) which is $[N/2]+1$ times
continuously differentiable in $\e^2$. Hence, we have the Taylor
polynomial~(\ref{4.11}). The point $(\xm^{\e,N},\xim^{\e,N})$ is the
maximizer for $u^{\e,N}$ since by hypothesis~H2 the Hessian of
$u^{\e,N}$ at this point differs from $2H_2$ by an error of order
$\Odr(\e^2)$. We employ this fact and expand $u^{\e,N}(x',\xi_n)$ in
Taylor series at $(\xm^{\e,N},\xim^{\e,N})$. As a result, we have
the estimate
\begin{equation}\label{4.13}
u^{\e,N}(x',\xi_n)-u^{\e,N}(\xm^{\e,N},\xim^{\e,N})\leqslant -C_1\left(
|x'-\xm^{\e,N}|^2+|\xi_n-\xim^{\e,N}|^2\right),
\end{equation}
where $C_1$ is a positive constant independent of $\e$, $x'$ and
$\xi_n$. This estimate is valid in a small fixed neighborhood $Q$ of
the point $(\xm,d(\xm)/2)$. Since the function $u^{\e,N}$ has the
maximum at $(\xm^{\e,N},\xim^{\e,N})$, we can choose the neighborhood
$Q$ so that outside it the estimate
\begin{equation}\label{4.14}
u^{\e,N}(x',\xi_n)-u^{\e,N}(\xm^{\e,N},\xim^{\e,N})\leqslant -C_2<0
\end{equation}
holds true, where the constant $C_2$ is independent of $\e$. Let us
choose $(x',\xi_n)$ so that
\begin{equation}\label{4.15}
|x'-\xm^{\e,N}|^2+|\xi_n-\xim^{\e,N}|^2\geqslant C_3\e^{2N+2},
\end{equation}
where $C_3$ is a positive constant independent of $\e$, $x'$ and
$\xi_n$. Then it follows from (\ref{4.13}), (\ref{4.14}) that for
such $(x',\xi_n)$ the inequality
\begin{equation*}
u^{\e,N}(x',\xi_n)\leqslant
u^{\e,N}(\xm^{\e,N},\xim^{\e,N})-C_1C_3\e^{2N+2}
\end{equation*}
is valid. Together with (\ref{4.12}) it implies that a maximizer
$(x_\e',\xi_\e)$ of $u^\e$ can not satisfy (\ref{4.15}) for
sufficiently small $\e$ and sufficiently large $C_3$ and therefore
\begin{equation}\label{4.16}
|x'_\e-\xm^{\e,N}|^2+|\xi_\e-\xim^{\e,N}|^2\leqslant C_3\e^{2N+2}.
\end{equation}
This inequality and (\ref{4.11}) prove (\ref{4.11a}).
\end{proof}

\subsection{The terms of order $\eps^{4}$}
In order to determine $\xmb$ and $\xi_{2}$ we shall need the terms
of order $\eps^{4}$ in the asymptotics of the gradient of
$u^{\eps}$, for which we need to consider $u_{4}$. We must also
develop $d$ and $p$ around $\xm$. In full generality, and to obtain
the full asymptotic expansion, these developments should be written
in terms of homogeneous polynomials of increasing degree. However,
to obtain the first two terms in the asymptotics
we will only need terms up to the homogeneous polynomials of third
degree. Therefore, we shall choose a form that will be more convenient for our
calculations. Write thus $d$ and $p$ as follows.
\begin{equation}\label{dpexp}
\begin{array}{lll}
d(x') = d_{0} +
\da^{t}(x'-\xm)+(x'-\xm)^{t}D_{2}(x'-\xm)+D_{3}(x'-\xm)+\Odr(|x'-\xm|^{4}),
\eqskip p(x') = p_{0} +
\pa^{t}(x'-\xm)+(x'-\xm)^{t}P_{2}(x'-\xm)+P_{3}(x'-\xm)+\Odr(|x'-\xm|^{4}),
\end{array}
\end{equation}
where $\da=\nabla_{x'}d(\xm)$, $\pa=\nabla_{x'}p(\xm)$, $2D_{2}$
and $2P_{2}$ are the Hessian matrices of $d$ and $p$ at the point $\xm$,
respectively, and $D_{3}$ and $P_{3}$
are homogeneous polynomials of degree three to be specified below.

Due to the relation between $d$ and $p$ via the functions $h_{\pm}$
and the fact that $H'(\xm)=h'_{-}(\xm)+h'_{+}(\xm)$ must vanish, we
easily obtain that
\[
\pa=-\frac{\ds 1}{\ds 2}d_{0}\da.
\]

In the case of $u_{4}$, the relevant derivatives are given by
\[
\begin{array}{lll}
\frac{\ds \partial u_{4}}{\ds \partial \xi_n}(x',\xi_{n}) & = &
-\frac{\ds 1}{\ds 2}\xi_{n}^{2}\Delta_{x'} d(x') -
\xi_{n}\Delta_{x'} p(x') +\alpha_{1}^{(4)}(x') \eqskip
\nabla_{x}u_{4}(x',\xi_{n}) & = & -\frac{\ds 1}{\ds
6}\xi_{n}^{3}\nabla_{x'}\left[\Delta_{x'} d(x')\right] -\frac{\ds
1}{\ds 2}\xi_{n}^{2}\nabla_{x'}\left[\Delta_{x'} p(x')\right]\eqskip
& &
\hspace*{1cm}+\xi_{n}\nabla_{x'}\alpha_{1}^{(4)}(x')+\nabla_{x'}\alpha_{0}^{(4)}(x').
\end{array}
\]
We shall first obtain the term of order $\eps^{4}$ in the derivative
of $u^{\eps}$ with respect to $\xi_{n}$. This will have a component
coming from the term of order $\eps^{2}$ in the corresponding
derivative of $u_{2}$, and another from the constant term in the
derivative of $u_{4}$. In the first case it is straightforward to
obtain that the required coefficient is given by
\begin{equation}\label{eps2u2}
-2\xi_{n2}+\da^{t}\xmb.
\end{equation}
In the case of $u_{4}$ the term coming from $\alpha_{1}^{(4)}$ is
given by
\[
\alpha_{1}^{(4)}(\xm+\Odr(\eps^2)) = \frac{\ds 1}{\ds
3}\left(d_{0}^2+p_{0}\right)
\tr(D_{2})+d_{0}\tr(P_{2})+\Odr(\eps^2).
\]
We thus obtain
\[
\begin{array}{lll}
\frac{\ds \partial u_{4}}{\ds
\partial\xi_{n}}(\xm+\Odr(\eps^2),\frac{\ds 1}{\ds
2}d_{0}+\Odr(\eps^2))& = & -\frac{\ds 1}{\ds
4}d_{0}^2\tr(D_{2})-d_{0}\tr(P_{2})\eqskip & &
\hspace*{0.5cm}+\frac{\ds 1}{\ds 3}\left(d_{0}^2+p_{0}\right)
\tr(D_{2}) \eqskip & &
\hspace*{1cm}+d_{0}\tr(P_{2})+\Odr(\eps^2)\eqskip &=& \frac{\ds
1}{\ds 12}\left( d_{0}^2+4p_{0}\right)\tr(D_{2})+\Odr(\eps^2)
\end{array}
\]
This, together with~(\ref{eps2u2}), yields
\begin{equation}
\label{graduxi}
\begin{array}{l}
\frac{\ds \partial u^{\eps}}{\ds \partial\xi_{n}}(\xm+\xmb\eps^{2}+\Odr(\eps^{3}),\frac{\ds 1}{\ds 2}d_{0}+
\xi_{n2}\eps^{2}+\Odr(\eps^{3})) =\eqskip
\hspace*{1.5cm}=\left[-2\xi_{n2}+\da^{t}\xmb+
\frac{\ds 1}{\ds 12}\left( d_{0}^2+4p_{0}\right)\tr(D_{2})\right]\eps^{4}+\Odr(\eps^{5}).
\end{array}
\end{equation}

We will now proceed to compute the gradient with respect to $x'$. The case of $u_{2}$ is
again straightforward and we obtain
\begin{equation}\label{gradu2x}
\begin{array}{l}
\nabla_{x'}u_{2}(\xm+\xmb\eps^{2}+\Odr(\eps^{3}),\frac{\ds 1}{\ds
2}d_{0}+\xi_{n2}\eps^{2} +\Odr(\eps^{3}))=\eqskip
\hspace*{0.7cm}=\left(\frac{\ds 1}{\ds
2}d_{0}+\xi_{n2}\eps^2\right)\left[\da+
2D_{2}\xmb\eps^2\right]-\frac{\ds 1}{\ds
2}d_{0}\da+2P_{2}\xmb\eps^2+\Odr(\eps^{3}) \eqskip
\hspace*{0.7cm}=\left(\xi_{n2}\da+d_{0}D_{2}\xmb
+2P_{2}\xmb\right)\eps^2+\Odr(\eps^{3}).
\end{array}
\end{equation}
In the case of $u_{4}$ we are only interested in the terms of order $\eps^{0}$. However, there
are now expressions of the form
\[
\nabla_{x'}\Delta_{x'}d(x') \mbox{ and } \nabla_{x'}\Delta_{x'}p(x'),
\]
this being the reason why we need the homogeneous polynomials of third degree in the expansions
of $d$ and $p$. On the other hand, this implies that the only relevant terms from
$D_{3}$ and $P_{3}$ are
those where one of the variables appears at least twice. If we write
\[
\begin{array}{lll}
D_{3}(x') & = & \dsum_{ijk}^{n-1}d_{ijk}x_{i}x_{j}x_{k}\eqskip
& = & \dsum_{i=1}^{n-1}\left[ d_{iii}x_{i}^{3}+
\dsum_{j=1,j\neq i}^{n-1}(d_{iij}+d_{iji}+d_{jii})x_{i}^2x_{j}\right]+r_{3}^{d}(x'),
\end{array}
\]
with
\[
r_{3}^{d}(x')=\dsum_{i\neq j\neq k}^{n-1}d_{ijk}x_{i}x_{j}x_{k},
\]
then we may assume without loss of generality that the coefficients $d_{ijk}$ are invariant
under any possible permutation of the indices. If we then denote $d_{iii}$ and
$d_{iij}=d_{iji}=d_{jii}$ ($i\neq j$) by $\delta_{ii}$ and $\delta_{ij}$, respectively,
the expression for $D_{3}$ becomes
\[
D_{3}(x') = \dsum_{i=1}^{n-1}\left(\delta_{ii}x_{i}^3+3\dsum_{j=1,j\neq i}^{n-1}\delta_{ij}x_{i}^2x_{j}
\right)
+ r_{3}^{d}(x').
\]
With this notation we get
\[
\begin{array}{lll}
\nabla_{x'}(\Delta_{x'}D_{3}(x')) & =
&6\nabla_{x'}\left[\dsum_{i=1}^{n-1}\left(\delta_{ii}x_{i}
+\dsum_{j=1,j\neq i}^{n-1}\delta_{ij}x_{j}\right)\right]
\eqskip & = &
6\left(\dsum_{j=1}^{n-1}\delta_{1j},\ldots,\dsum_{j=1}^{n-1}\delta_{n-1,j}\right).
\end{array}
\] In a similar fashion, if we write
\[
P_{3}(x') = \dsum_{i=1}^{n-1}\left(\pi_{ii}x_{i}^3+3\dsum_{j=1,j\neq i}^{n-1}\pi_{ij}x_{i}^2x_{j}
\right)
+ r_{3}^{p}(x'),
\]
we get
\[
\begin{array}{lll} \nabla_{x'}(\Delta_{x'}P_{3}(x')) & =
&6\nabla_{x'}\left[\dsum_{i=1}^{n-1}\left(\pi_{ii}x_{i}
+\dsum_{j=1,j\neq i}^{n-1}\pi_{ij}x_{j}\right)\right]\eqskip & = &
6\left(\dsum_{j=1}^{n-1}\pi_{1j},\ldots,\dsum_{j=1}^{n-1}\pi_{n-1,j}\right).
\end{array}
\]
In this way, we obtain after some lengthy but straightforward calculations,
\[
\begin{array}{r}
\nabla_{x'}u_{4}(\xm+\xmb\eps^2+\Odr(\eps^3),\frac{\ds 1}{\ds
2}d_{0}+\xi_{n2}\eps^2+\Odr(\eps^3)) = -\left(\frac{\ds d_0}{\ds
2}\right)\Si_\d-3\left(\frac{\ds d_0}{\ds 2}\right)^2\Si_\pi \eqskip
+\frac{\ds d_0}{\ds 2}\left[ (d_0^2+p_0)\Si_\d+\frac{\ds \tr
D_2}{\ds 3} (2d_0d_{1}-\frac{\ds 1}{\ds
2}d_{0}d_{1})+3d_0\Si_\pi+\tr P_2 d_{1}\right] \eqskip
+d_0p_0\Si_\d+ \frac{\ds \tr D_2}{\ds 3} (-\frac{\ds 1}{\ds 2}d_0^2
d_{1}+p_0d_{1})+3 p_0\Si_\pi-\frac{\ds \tr P_2}{\ds 2}
d_{0}d_{1}+\Odr(\e^2)
\eqskip
= \frac{\ds 1}{\ds
4}(d_0^2+4p_0)\left[3\Si_\pi +\frac{\ds 3d_0}{\ds 2}\Si_\d+
\frac{\ds \tr D_2}{\ds 3} d_1\right]+\Odr(\e^2)
\end{array}
\] where
\[
\Sigma_{\delta}
=
\left(\dsum_{j=1}^{n-1}\delta_{j1},\ldots,\dsum_{j=1}^{n-1}\delta_{j,n-1}\right)
\mbox{
and
}
\Sigma_{\pi}
=
\left(\dsum_{j=1}^{n-1}\pi_{j1},\ldots,\dsum_{j=1}^{n-1}\pi_{j,n-1}\right).
\]
Combining this with~(\ref{gradu2x}) yields
\begin{equation}\label{gradux}
\begin{array}{l}
\nabla_{x'}u^{\eps}(\xm+\xmb\eps^2+\Odr(\eps^3),\frac{\ds 1}{\ds
2}d_{0} +\xi_{n2}\eps^2+\Odr(\eps^3))= \eqskip\hspace*{1.5cm} =
\Big[\xi_{n2}\da+d_{0}D_{2}\xmb +2P_{2}\xmb \eqskip\hspace*{2.5cm} +
\frac{\ds 1}{\ds 4}(d_0^2+4p_0)\left[3\Si_\pi +\frac{\ds 3}{\ds
2}d_0\Si_\d+ \frac{\ds \tr D_2}{\ds 3} d_1\right]\Big]\eps^4
\eqskip\hspace*{3.5cm}+ \Odr(\eps^{5}),
\end{array}
\end{equation}
from which we obtain the second equation for $\xmb$ and $\xi_{n2}$
by equating the coefficient of $\eps^4$ to zero. From
equation~(\ref{graduxi}) we get
\begin{equation}\label{xin21}
\xi_{n2} = \frac{\ds 1}{\ds 2}\da^{t}\xmb + \frac{\ds 1}{\ds 24}\left(d_{0}^{2}
+4p_{0}\right)\tr D_{2}.
\end{equation}
Substituting this into equation~(\ref{gradux}) yields
\[
\begin{array}{lll}
 \left[\frac{\ds 1}{\ds 2}\left(
\da\da^{t}\right)+2P_{2}+d_{0}D_{2}\right]\xmb & = & -\frac{\ds
1}{\ds 4}\left(d_{0}^2+4p_{0}\right)\left( \frac{\ds \tr D_{2}}{\ds
2}d_{1} +3\Si_\pi +\frac{\ds 3}{\ds 2}d_0\Si_\d\right).
\end{array}
\]
To prove that there is a unique solution, we must show that
the matrix multiplying $\xmb$ is nonsingular. In order to do this, we shall relate the
terms appearing above to those in the expansions of the functions $h_{\pm}$ and $H$. If we
write
\[
h_{\pm}(x') = h_{0}^{\pm} + \left(h_{1}^{\pm}\right)^{t}(x'-\xm) +
(x'-\xm)^{t}H_{2}^{\pm}(x'-\xm) + \Odr(|x'-\xm|^{3}),
\]
we see that
\[
\begin{array}{l}
d_{0} = h_{0}^{+}-h_{0}^{-}\eqskip \da = h_{1}^{+}-h_{1}^{-}\eqskip
D_{2} = H_{2}^{+}-H_{2}^{-} \eqskip H_2=H_2^++H_2^-
\end{array}
\]
and
\[
\begin{array}{l}
p_{0} = h_{0}^{+}h_{0}^{-}\eqskip
\pa = h_{0}^{+}h_{1}^{-}+h_{0}^{-}h_{1}^{+}\eqskip
P_{2} = h_{0}^{+}H_{2}^{-}+h_{0}^{-}H_{2}^{+}+h_{1}^{+}(h_{1}^{-})^{t}.
\end{array}
\]
Replacing this in the expression above yields, after some manipulation
\[
\begin{array}{lll}
d_{0}D_{2}+2P_{2}+\frac{\ds 1}{\ds 2}\da\da^{t}&=&
H_{0}H_{2}+\frac{\ds 1}{\ds 2}(h_{1}^{+}-h_{1}^{-})(h_{1}^{+}-h_{1}^{-})^{t} +h_{1}^{+}(h_{1}^{-})^{t}\eqskip
& = & H_{0}H_{2},
\end{array}
\]
where we used the fact that $0=H_{1}=h_{1}^{+}+h_{1}^{-}$. Since the matrix
$H_{0}H_{2}$ is negative definite by hypothesis, we may invert it to obtain
\begin{equation}
\label{x2}
\begin{array}{lll}
\xmb & = & -\frac{\ds 1}{\ds 4} H_{0}\left[\frac{\ds \tr D_{2}}{\ds
2}H_{2}^{-1}d_{1} +3H_{2}^{-1}\left(\Sigma_{\pi}+\frac{\ds 1}{\ds
2}d_{0}\Sigma_{\delta}\right)\right],
\end{array}
\end{equation}
where we have used the fact that $d_{0}^{2}+4p_{0}= H_{0}^2$.
Plugging this back into~(\ref{xin21}) yields
\begin{equation}
\label{xin2}
\begin{array}{lll}
\xi_{n2} & = & -\frac{\ds 1}{\ds 8}H_{0}\left[\frac{\ds 1}{\ds 2}\tr D_{2}\da^{t}H_{2}^{-1}\da
+3\da^{t}H_{2}^{-1}\left(\Sigma_{\pi}+\frac{\ds 1}{\ds 2}d_{0}\Sigma_{\delta}\right)\right]
\eqskip
& & \hspace*{1cm}+\frac{\ds 1}{\ds 24}H_{0}^2\tr D_{2}.
\end{array}
\end{equation}
If we now evaluate $u_{2}$ and $u_{4}$ at the maximizer we obtain
\[
u_{2}\left(\xm+\xmb\eps^{2}+\Odr(\eps^{3}),\frac{\ds 1}{\ds
2}d_{0}+\xi_{n2}\eps^{2}+\Odr(\eps^{3})\right) = \frac{\ds 1}{\ds 4}
H_{0}^{2} + \Odr(\eps^{3})
\]
and
\begin{align*}
u_{4}&\left(\xm+\xmb\eps^{2}+\Odr(\eps^{3}),\frac{\ds 1}{\ds
2}d_{0}+\xi_{n2}\eps^{2}+\Odr(\eps^{3})\right)
\\
&\hphantom{\xm+\xmb\eps^{2}+\Odr(\eps^{3}),}=\frac{\ds 1}{\ds
8}H_{0}^{2}\left[ d_{0}\tr(D_{2})+2\tr(P_{2})\right] + \Odr(\eps^2),
\end{align*}
where $\xi_{n2}$ and $\xmb$ are given as above. We now use the
fact that $p(x')=[H^{2}(x')-d^{2}(x')]/4$ and $2\tr D_{2}$ and $2\tr P_{2}$
are the Laplacian of $d$ and $p$, respectively, to rewrite $d_{0}\tr(D_{2})+2\tr(P_{2})$
as
\[
\frac{\ds 1}{\ds 2}\Delta(H^2_{0})-|\nabla d_{0}|^2.
\]
We have thus proven the following
\begin{theorem}\label{xm2xin2}
Under the conditions of Theorem~\ref{maxthm}, the maximizer $(x_{*}',\xi_{n*})$ of $u^{\eps}$
given by Lemma~\ref{lemmaeps2}
satisfies the asymptotic expansion
\[
(x_{*}',\xi_{n*})= (\xm + \xmb \eps^{2},\frac{\ds 1}{\ds 2}d_{0}+\xi_{n2}\eps^{2})
+\Odr(\eps^{3})
\mbox{ as } \eps\to 0,
\]
where $\xmb$ and $\xi_{n2}$ are given by~(\ref{x2}) and~(\ref{xin2}), respectively. The
corresponding maximum satisfies the asymptotic expansion given in Theorem~\ref{maxthm}.
\end{theorem}

\section{Proof of Theorem~\ref{torsrig}\label{storsrig}}
To prove Theorem~\ref{torsrig} we need only to integrate $u^{\eps}$
in $\Omega_{\eps}$. More precisely, using the expression given
by Theorem~\ref{th1.1}, we have to compute
\[
\begin{array}{lll}
\dint_{\Omega_{\eps}} u^{\eps}(x)\;{\rm d}x & = & \dint_{\omega}
\dint_{-\eps h_{-}(x')}^{\eps h_{+}(x')}\eps^2
u_{2}\left(x',\frac{\ds x_{n}}{\ds \eps}\right)
+\eps^{4}u_{4}\left(x',\frac{\ds x_{n}}{\ds \eps}\right)\eqskip & &
\hspace*{0.5cm}+ \eps^{6}u_{6}\left(x',\frac{\ds x_{n}}{\ds
\eps}\right)+ \Odr(\eps^{7})\;{\rm d}x_{n}{\rm d}x' \eqskip & = &
\eps^{3}\dint_{\omega}\dint_{-h_{-}(x')}^{h_{+}(x')}u_{2}(x',\xi_{n})\;{\rm
d}\xi_{n}{\rm d}x' \eqskip & &
\hspace*{0.5cm}+\eps^{5}\dint_{\omega}\dint_{-h_{-}(x')}^{h_{+}(x')}
u_{4}(x',\xi_{n})\;{\rm d}\xi_{n}{\rm d}x' \eqskip & &
\hspace*{1cm}+\eps^{7}\dint_{\omega}\dint_{-h_{-}(x')}^{h_{+}(x')}
u_{6}(x',\xi_{n})\;{\rm d}\xi_{n}{\rm d}x' + \Odr(\eps^{8}).
\end{array}
\]
We shall now compute these three integrals separately. Using the expression for $u_{2}$ we
have
\[
\begin{array}{lll}
\dint_{-h_{-}(x')}^{h_{+}(x')}u_{2}(x',\xi_{n})\;{\rm d}\xi_{n} & = &
\dint_{-h_{-}(x')}^{h_{+}(x')}-\xi_{n}^{2}+\xi_{n}d(x')
+p(x')\;{\rm d}\xi_{n}\eqskip
& = & \left.-\frac{\ds 1}{\ds 3}\xi_{n}^{3}+\frac{\ds 1}{\ds 2}\xi_{n}^{2}d(x')
+p(x')\xi_{n}\right]_{-h_{-}(x')}^{h_{+}(x')}\eqskip
& = & -\frac{\ds 1}{\ds 3}\left[ h_{+}^{3}(x')+h_{-}^{3}(x')\right]
\eqskip
& & \hspace*{0.5cm}
+ \frac{\ds 1}{\ds 2}\left[ h_{+}^{2}(x')-h_{-}^{2}(x')\right]d(x')
\eqskip
& & \hspace*{1cm}
+\left[ h_{+}(x')+h_{-}(x')\right]p(x')
\end{array}
\]
and, taking into account the expressions for both $d$ and $p$, we get that the above
equals
\[
\frac{\ds 1}{\ds 6}\left[ h_{+}^{3}(x')+h_{-}^{3}(x')+3h_{+}^{2}(x')h_{-}(x')+
3h_{+}(x')h_{-}^{2}(x')\right]=\frac{\ds 1}{\ds 6}H^{3}(x')
\]
as desired.

For $u_{4}$ we have
\[
\begin{array}{lll}
\dint_{-h_{-}(x')}^{h_{+}(x')}u_{4}(x',\xi_{n})\;{\rm d}\xi_{n} & = &
\dint_{-h_{-}(x')}^{h_{+}(x')}-\frac{\ds 1}{\ds 6}\xi_{n}^{3}\Delta_{x'}d(x')
-\frac{\ds 1}{\ds 2}\xi_{n}^{2}\Delta_{x'}p(x')
\eqskip
& & \hspace*{0.5cm}
+\alpha_{1}^{(4)}(x')\xi_{n}+\alpha_{0}^{(4)}(x')\;{\rm d}\xi_{n}
\eqskip
& = &
-\frac{\ds 1}{\ds 24}\xi_{n}^{4}\Delta_{x'}d(x')
-\frac{\ds 1}{\ds 6}\xi_{n}^{3}\Delta_{x'}p(x')
\eqskip
& & \hspace*{0.5cm}
\left.
+\frac{\ds 1}{\ds 2}\alpha_{1}^{(4)}(x')\xi_{n}^{2}+\alpha_{0}^{(4)}(x')\xi_{n}
\right]_{-h_{-}(x')}^{h_{+}(x')}.
\end{array}
\]
Using the expressions for $\alpha_{0}^{(4)}$ and $\alpha_{1}^{(4)}$ we see that
the above integral equals
\[
\begin{array}{l}
-\frac{\ds 1}{\ds 24}\left[h_{+}^{4}(x')-h_{-}^{4}(x')\right]\Delta_{x'}d(x')
-\frac{\ds 1}{\ds 6}\left[h_{+}^{3}(x')+h_{-}^{3}(x')\right]\Delta_{x'}p(x')
\eqskip
\hspace*{0.5cm}
+\frac{\ds 1}{\ds 12}\left[(d^{2}(x')+p(x'))\Delta_{x'}d(x')+3d(x')\Delta_{x'}p(x')\right]
\left[ h_{+}^{2}(x')-h_{-}^{2}(x')\right]\eqskip
\hspace*{1cm}
+\frac{\ds 1}{\ds 6}\left[ d(x')p(x')\Delta_{x'}d(x')+3p(x')\Delta_{x'}p(x')\right]
\left[ h_{+}(x')+h_{-}(x')\right]\eqskip
= \frac{\ds 1}{\ds 24}\left[h_{+}^{4}(x')-h_{-}^{4}(x')+2h_{+}^{3}(x')h_{-}(x')
-2h_{+}(x')h_{-}^{3}(x')\right]\Delta_{x'}d(x')
\eqskip
\hspace*{0.5cm}
+\frac{\ds 1}{\ds 12}\left[ h_{+}^{3}(x')+h_{-}^{3}(x')+3h_{+}^{2}(x')h_{-}(x')
+3h_{+}(x')h_{-}^{2}(x')\right]\Delta_{x'}p(x')\eqskip
=\frac{\ds 1}{\ds 24}\left[ d^{3}(x')H(x')+4h_{+}^{3}(x')h_{-}(x')-
4h_{+}(x')h_{-}^{3}(x')\right]\Delta_{x'}d(x')
\eqskip
\hspace*{0.5cm}
\frac{\ds 1}{\ds 12}H^{3}(x')\Delta_{x'}p(x')\eqskip
=\frac{\ds 1}{\ds 24}d(x')H(x')\left[ d^{2}(x')+4p(x')\right]\Delta_{x'}d(x')
+\frac{\ds 1}{\ds 12}H^{3}(x')\Delta_{x'}p(x')\eqskip
=\frac{\ds 1}{\ds 24}H^{3}(x')\left[ d(x')\Delta_{x'}d(x')+2\Delta_{x'}p(x')\right].
\end{array}
\]
The coefficient in the Theorem is now obained by using the relations between $p$ and $H$
and $d$.

Although the expression for $u_{6}$ is much more involved, the necessary computations
are similar to those above. We need to compute
\[
\begin{array}{lll}
\dint_{-h_{-}(x')}^{h_{+}(x')}u_{6}(x',\xi_{n})\;{\rm d}\xi_{n} & = &
\dint_{-h_{-}(x')}^{h_{+}(x')} \frac{\ds \xi_{n}^{5}}{\ds 120}\D_{x'}^{2} d(x')+ \frac{\ds \xi_{n}^{4}}{\ds 24}\D_{x'}^{2} p(x')\eqskip
& & \hspace*{0.2cm}-\frac{\ds \xi_{n}^{3}}{\ds 36}\D_{x'}\left\{\left[d^{2}(x')+p(x')\right]\D_{x'}d(x')+
3d(x')\D_{x'}p(x')\right\}\eqskip
& & \hspace*{0.4cm} -\frac{\ds \xi_{n}^{2}}{\ds 12}\D_{x'}\left[ d(x')p(x')\D_{x'}d(x')+3p(x')\D_{x'}p(x')\right]\eqskip
& &  \hspace*{0.6cm} +\alpha^{(6)}_{1}(x')\xi_{n}+\alpha^{(6)}_{0}(x')\;{\rm d}\xi_{n}.
\end{array}
\]
Using now the following identities
\[
\begin{array}{lll}
h_{+}^3(x')+h_{-}^3(x') & = & H^3(x')-3p(x')H(x')\eqskip
h_{+}^4(x')-h_{-}^4(x') & = & H^3(x')d(x')-2d(x')p(x')H(x')\eqskip
h_{+}^5(x')+h_{-}^5(x') & = & H^3(x')d^2(x')+p(x')H(x')\left[p(x')-d^2(x')\right]\eqskip
h_{+}^6(x')-h_{-}^6(x') & = & H^3(x')d^3(x')+3d(x')p^2(x')H(x'),
\end{array}
\]
after some computations we obtain
\[
\begin{array}{lll}
\dint_{-h_{-}(x')}^{h_{+}(x')}u_{6}(x',\xi_{n})\;{\rm d}\xi_{n} & = &
\frac{\ds 1}{\ds 720}\left[H^3(x')d^3(x')\right.\eqskip
& & \hspace*{0.5cm}\left.+3d(x')p^2(x')H(x')\right]\D_{x'}^{2} d(x')
\eqskip
& & + \frac{\ds 1}{\ds 120}\left\{H^3(x')d^2(x')\right.\eqskip
& & \hspace*{0.5cm}\left.+p(x')H(x')\left[p(x')-d^2(x')\right]\right\}
\D_{x'}^{2} p(x')\eqskip
& & -\frac{\ds 1}{\ds 144}\left[H^3(x')d(x')-2d(x')p(x')H(x')\right]\eqskip
& & \hspace*{0.5cm}\times\D_{x'}\left\{\left[d^{2}(x')+p(x')\right]\D_{x'}d(x')+
3d(x')\D_{x'}p(x')\right\}\eqskip
& & -\frac{\ds 1}{\ds 36}\left[H^3(x')-3p(x')H(x')\right]\eqskip
& & \hspace*{0.5cm}\times\D_{x'}\left[ d(x')p(x')\D_{x'}d(x')+3p(x')
\D_{x'}p(x')\right]\eqskip
& & + \frac{\ds 1}{\ds 2}d(x') H(x')\alpha_{1}^{(6)}(x') + H(x')\alpha^{(6)}_{0}(x').
\end{array}
\]

This, together with the expressions for the integrals of $u_{2}$ and $u_{4}$ given above yields
the formula in Theorem~\ref{torsrig}.

\section{Examples\label{sexamp}}
Let us now apply our results to some special cases in order to test the accuracy
of the approximations for concrete examples. As may be seen from Figures~\ref{fig:folium}
and~\ref{fig:lemn} below, although the error for either the $L^{2}-$ norm or for the torsion
stays below $5\%$ for $\e$ up to $0.6$, this can vary substantially as the scaling parameter
approaches one. For the examples considered below the error at $\e$ equal to one for the torsion, for
instance, varies between less than $2\%$ and $100\%$ in the cases of the folium and the disc, respectively.
The reason for this is simply that, even in the case where a Taylor (or Laurent) series exists for the quantities
under consideration, the series expansion for these quantities will have a specific radius of convergence.
In the case of the disc, for instance, we have that the torsion is given by
\[
\frac{\ds \pi}{\ds 2}\frac{\ds \e^2}{\ds 1+\e^2}
\]
which has a radius of convergence of one, thus explaining the large error found in this case.

\subsection{Descartes's folium}
We consider the case of the domain defined by
\[
\Om_\e=\left\{ (x,y)\in\R^{2}: x(x^2+3\e^{-2}y^2)-x^2+\e^{-2}y^2<0,
\ 0<x<1\right\},
\]
where we have chosen coordinates in such a way that the scaling is done along
the $y-$axis and the corresponding height with respect to this axis
is minimal. In this case we then have
\[
\begin{array}{lll}
H(x)=2h_{+}(x)=2h_{-}(x)=2x\sqrt{\frac{\ds 1-x}{\ds 1+3x}}, & d(x)\equiv 0, &
p(x) = \frac{\ds (1-x)x^2}{\ds 3x+1},
\end{array}
\]
yielding
\begin{equation}\label{udesc}
\begin{array}{lll}
u^\e(x,\xi) & = & \left[\frac{\ds (1-x)x^2}{\ds
1+3x}-\xi^2\right]\e^2\eqskip & & \hspace*{0.2cm}+ \frac{\ds
\left[(3x+1)^3-4\right]\left[(x-1)x^2+(3x+1)\xi^2\right]}{\ds
3(1+3x^4)}\e^4 \eqskip & & \hspace*{0.4cm}-\frac{\ds
\left[(x-1)x^2+(3x+1)\xi^2\right]p_{1}(x,\xi) }{\ds (3x+1)^7}\e^6 +
\Odr(\eps^{8}),
\end{array}
\end{equation}
where
\[
p_{1}(x,\xi):=1-30x+51x^2+48x^3+135x^4+162x^5+81x^6-12\xi^2-36x\xi^2.
\]

The direct application of Theorem~\ref{maxthm}, where we only have the explicit formula for terms
up to order $\eps^4$, yields
\[
\max_{x\in\Omega_{\e}} u^{\e}(x,y) = \frac{\ds 1}{\ds 9}(2\sqrt{3}-3)\e^2+
\frac{\ds 1}{\ds 9}(12 - 7\sqrt{3})\e^4+\Odr(\e^{5}).
\]
Note that in this case determining the maximum directly from~(\ref{udesc}) in order to obtain
a better approximation implies solving an algebraic equation of degree nine. Actually, even solving explicitly for the maximizer using $u$ up to order $\e^4$ and taking into consideration that we know beforehand using symmetry that $\xi$ must vanish, implies having to solve an algebraic equation of degree five.

Computing the expansion for the torsion will in turn yield
\begin{equation}\label{foltors}
\begin{array}{lll}
\e\dint_{\Om} u(x,\xi)dxd\xi & = &\left(\frac{\ds 16\pi}{\ds
243\sqrt{3}}-\frac{\ds 1}{\ds 9}\right) \e^{3}-\left(\frac{\ds
16\pi}{\ds 243\sqrt{3}}-\frac{\ds 37}{\ds 315}\right)\e^{5}\eqskip &
& \hspace*{0.2cm}+\left(\frac{\ds 80\pi}{\ds 2187\sqrt{3}}-\frac{\ds
593}{\ds 9009}\right) \e^{7}+ \Odr(\eps^{9}).
\end{array}
\end{equation}
In Figure~\ref{fig:folium} we show, for various values of $\eps$, the
relative errors for the $L^{2}$ norm and for the torsion of the diference
between the values of the numerical solution determined using the method of
fundamental solutions (MFS) and the asymptotic expansions given by~(\ref{udesc}) and~(\ref{foltors}). We note that the error at $\eps$ equal to one -- which
corresponds to the actual folium -- is of the order of $3.5\%$ and $2\%$
respectively.

\begin{figure}[!ht]\label{fig:folium}
\centering
\includegraphics[width=1\textwidth]{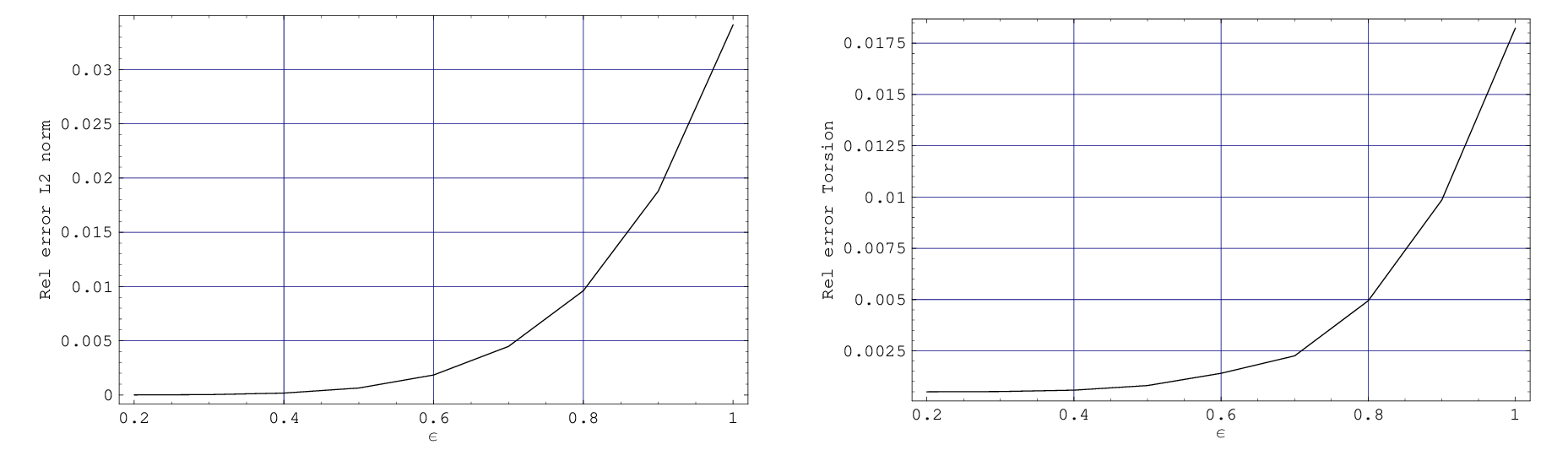}
\caption{Graphs of the relative errors for the $L^{2}$ norm and the torsion
in the case of the folium; the comparison is between the numerical solution
obtained by the MFS and the asymptotic expansion given by equations~(\ref{udesc}) and~(\ref{foltors}).}
\end{figure}

\subsection{Lemniscate}
We consider the domain $\Om_\e$ whose boundary is the lemniscate
defined by
\[
\left( x^2+ \e^{-2}y^2\right)^2 = x^2-\e^{-2}y^2, \quad x>0.
\]
The functions $H$, $h_\pm$, $d$ and $p$ are given by the formulas
\begin{align*}
&H(x)=2h_{+}(x)=2h_{-}(x) = 2\left[-\frac{1}{2}-x^2+
\frac{1}{2}(1+8x^2)^{1/2}\right]^{1/2},
\\
&d(x)\equiv 0,\quad p=-\frac{1}{2}-x^2+
\frac{1}{2}(1+8x^2)^{1/2},\quad \om=(0,1).
\end{align*}
Some straightforward calculations give
\begin{align*}
u^\e(x,\xi)&=\left(-\frac{1}{2}-x^2+
\frac{1}{2}(1+8x^2)^{1/2}-\xi^2\right)\e^2+
\\
&+\frac{\left[(1+8x^2)^{3/2}-2\right]\left(1+2x^2-\sqrt{1+8x^2}+ 2\xi^2\right)}
{\ds 2(1+8x^2)^{3/2}}\e^4
\\
&+\Bigg(4\xi^4 \frac{32x^2-1}{(1+8x^2)^{7/2}}
\\
&+\xi^2 \frac{ (-512
x^6-192x^4-216x^2+7)\sqrt{1+8x^2}+256x^4+328x^2-8}{(1+8x^2)^{7/2}}
\\
&- \frac{\big(1+2x^2-\sqrt{1+8x^2}\big) p_2(x) }{2(1+8x^2)^{7/2}}
\Bigg)\e^6+\Odr(\e^8),
\\
p_2(x)&:=\sqrt{1+8x^2}(512x^6 +192x^4+152x^2-5)-128x^4-268x^2+6.
\end{align*}
The maximum of $H$ is now situated at $\xm=\sqrt{3}/(2\sqrt{2})$ and using
\begin{align*}
& H_0=\frac{\sqrt{2}}{2}, \quad d_0=\tr D_2=0,\quad 2\tr
P_2=-\frac{3}{2},
\end{align*}
in Theorem~\ref{maxthm} then yields
\begin{align*}
& \max_{x\in\Omega_{\eps}} u^{\eps}(x) =
\frac{\e^2}{8}-\frac{3\e^4}{32} +\Odr(\e^5).
\end{align*}
By applying Theorem~\ref{torsrig} we arrive at the asymptotics for
the torsional rigidity
\begin{align*}
\int_{\Om_\e} u^\e(x)\di x=& \frac{3\pi-8}{48}\e^3+ \left[
\frac{\sqrt{3}}{4}\ln(2+\sqrt{3})-\frac{3\pi}{16}\right]\e^5
\\
&+ \left[ \frac{13}{18}+ \frac{5\pi}{16} -
\frac{20\sqrt{3}}{27}\ln(2+\sqrt{3}) \right] \e^7+\Odr(\e^9),
\end{align*}
where the integrals appearing in the coefficients can be calculated
by the Euler substitution
\begin{equation*}
\sqrt{1+8x^2}=xt+1,\quad x=\frac{2t}{8-t^2},\quad
t=\frac{\sqrt{1+8x^2}-1}{x}.
\end{equation*}

As in the previous example, we show in Figure~\ref{fig:lemn} the relative
errors for the $L^2$ norm and the torsion in this case. However, comparing the
two examples gives that for $\e$ larger than approximately $0.4$ the errors
become much larger than in the case of the folium.

\begin{figure}[!ht]\label{fig:lemn}
\centering
\includegraphics[width=1\textwidth]{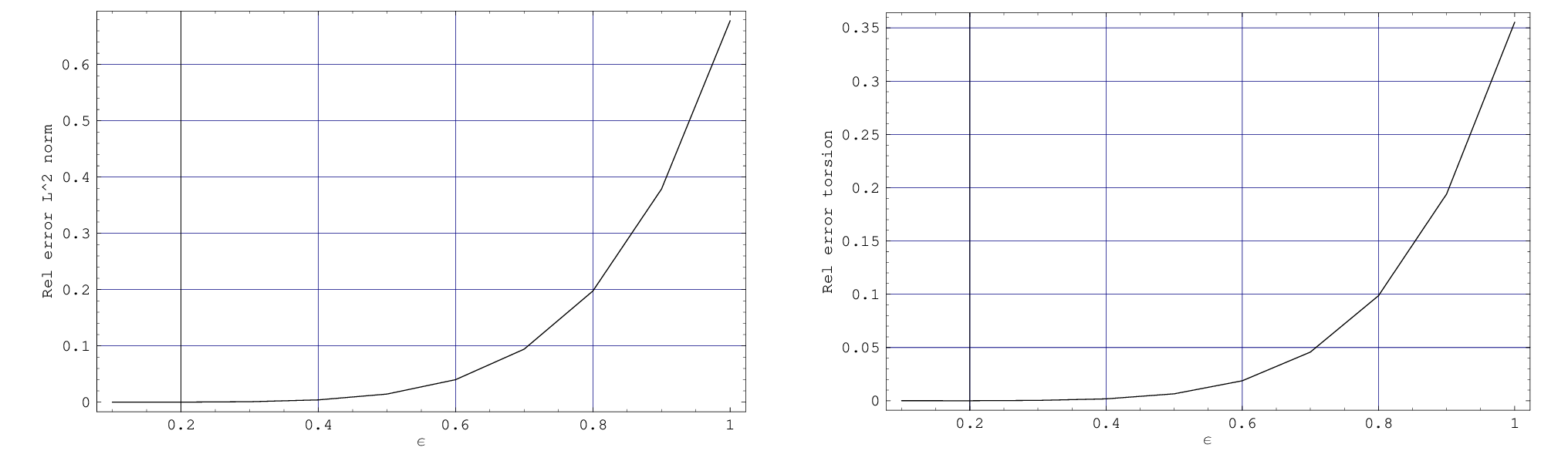}
\caption{Graphs of the relative errors for the $L^{2}$ norm and the torsion
in the case of the lemniscate, with the different quantities computed in a
similar way to what was done for the folium.}
\end{figure}

\subsection{Ellipsoids} As mentioned in the Introduction, these are one of the
few examples where the explicit solution of the corresponding equation~(\ref{ellipteq})
is known. More precisely, if we consider ellipsoids defined by
\[
E = \left\{ x\in\R^{n}: \left(\frac{\ds x_{1}}{\ds a_{1}}\right)^2+\dots+
\left(\frac{\ds x_{n}}{\ds a_{n}}\right)^2=1\right\}
\]
we have that the solution of equation~(\ref{ellipteq}) in this case is given by
\[
u(x) = \frac{\ds 1}{\ds \frac{\ds 1}{\ds a_{1}^2}+\dots+\frac{\ds 1}{\ds a_{n}^2}}
\left[1-\left(\frac{\ds x_{1}}{\ds a_{1}}\right)^2-\dots-
\left(\frac{\ds x_{n}}{\ds a_{n}}\right)^2\right]
\]
The maximum is thus localized at the origin and is given by
\[
M = \frac{\ds 1}{\ds \frac{\ds 1}{\ds a_{1}^2}+\dots+\frac{\ds 1}{\ds a_{n}^2}}.
\]
Let's assume that $a_{n}$ is the smallest of the $a_{j}$'s, and we thus pick this
direction to be that along which we scale the domain. We then have
\[
 H(x') = 2h_{+}(x') = 2h_{-}(x') = 2a_{n}\left[1-\left(\frac{\ds x_{1}}{\ds a_{1}}\right)^2-\dots-
\left(\frac{\ds x_{n-1}}{\ds a_{n-1}}\right)^2\right]^{1/2},
\]
while $d$ vanishes and
\[
p(x') = a_{n}^2\left[1-\left(\frac{\ds x_{1}}{\ds a_{1}}\right)^2-\dots-
\left(\frac{\ds x_{n-1}}{\ds a_{n-1}}\right)^2\right].
\]
 From this it follows that the maximizer is $\Odr(\eps^{3})$, while the maximum has the expansion
\[
M = a_{n}^2\eps^{2}-a_{n}^4\dsum_{j=1}^{n-1}
\frac{\ds 1}{\ds a_{j}^2} \eps^{4} + \Odr(\eps^{6}).
\]
A straightforward analysis of the error
\[
\mathcal{E}_{M} = \left[\frac{\ds 1}{\ds \frac{\ds 1}{\ds a_{1}^2}+\dots+\frac{\ds 1}{\ds \e^{2}a_{n}^2}}-
a_{n}^2\eps^{2}+a_{n}^4\dsum_{j=1}^{n-1}
\frac{\ds 1}{\ds a_{j}^2} \eps^{4}\right]
\left(\frac{\ds 1}{\ds \frac{\ds 1}{\ds a_{1}^2}+\dots+\frac{\ds 1}{\ds \e^{2} a_{n}^2}}
\right)^{-1}
\]
yields that this satisfies
\[
0\leq\mathcal{E}_{M}\leq \eps^{4}(n-1)^2,
\]
with equality on the right-hand side being achieved for the ball.

For the sake of comparison with the previous two-dimensional
examples, we shall now consider the case of the planar disc (in the
above notation, $n=2, a_{1}=1, a_{2}=\e$). In this case the above
expressions yield that the relative error of the $L^{2}$ norm and of
the torsion are, respectively, $\e^{12} + \Odr(\eps^{13})$ and
$\e^{6}+\Odr(\eps^{7})$. Thus, although the approximation is
very good for sufficiently small values of $\e$, the error does
become quite large as $\e$ reaches one.

\section*{Acknowledgments} D.B. was partially supported by RFBR
(10-01-00118), by the grants of the President of Russia for young
scientists (MD-453.2010.1) and for Leading Scientific Schools
(NSh-6249.2010.1), and by the Federal Task Program (contract 02.740.11.0612).
P.F. was partially supported by POCTI/POCI2010 and PTDC/MAT/101007/2008,
Portugal. Part of this work was done while P.F. was visiting the
Erwin Schr\"{o}dinger Institute in Vienna within the scope of the
program {\it Selected topics in spectral theory} and he would like
to thank the organizers B. Hellfer, T. Hoffman-Ostenhof and A.
Laptev for their hospitality and ESI for financial support. He would
also like to thank Rodrigo Ba\~{n}uelos for some very helpful
conversations during this period. We would also like to thank P.
Antunes for having carried out the numerical results used in
Section~\ref{sexamp}.

\end{document}